\documentclass[10pt]{amsart}
\begin{document}
\title{Closures of $O_n$-orbits in the flag variety for $GL_n$}
\author{William M. McGovern}
\subjclass{22E47,57S25}
\keywords{flag variety, pattern avoidance, rational smoothness}
\begin{abstract}
We give a pattern avoidance criterion for the conjugates of the bottom vertex of the Bruhat graph attached to an $O_n$ orbit $\mathcal O$ in the flag variety of $GL_n$ to have degree equal to the rank of this graph as a poset.  This condition is known to be necessary for rational smoothness of $\overline{\mathcal O}$; we conjecture that it is also sufficient, proving that this holds for $n$ even.
\end{abstract}
\maketitle

\section{Introduction}
Let $G$ be a complex reductive group with Borel subgroup $B$ and let $K=G^\theta$ be the fixed point subgroup of an involution of $G$.  In this paper we develop the program begun in \cite{M09} and continued in \cite{MT09,M10}, seeking to characterize the $K$-orbits in $G/B$ with rationally smooth closure via a combinatorial criterion.  Here we treat the case $G = GL_n\mathbb C, K = O_n\mathbb C$ and give a necessary pattern avoidance criterion, conjecturing that it is also sufficent, and proving this for $n$ even.  We focus on the order ideal of orbits with closures contained in a fixed one $\bar{\mathcal O}$; this is an interval inside the poset of all $K$-orbits in $G/B$, ordered by containment of closures.  Richardson and Springer have defined an action of the braid monoid corresponding to the Weyl group $W$ of $G$ on this poset \cite{RS90}, which we use to make it and its order ideals into graphs.  Our starting point is a necessary graph-theoretic criterion for rational smoothness inspired by the work of Carrell and Peterson on Schubert varieties \cite{C94}.  This criterion has been generalized by Hultman to a necessary and sufficient criterion for rational smoothness for $G=GL_{2n}\mathbb C, K=$Sp$_{2n}\mathbb C$ \cite{H12}; it is {\sl not} in general sufficient in our situation.

I would like to warmly thank Axel Hultman for suggesting the approach used in this paper.

\section{Preliminaries}
Set $G = GL_n\mathbb C, K = O_n\mathbb C$.  Let $B$ be the subgroup of upper triangular matrices in $G$.  The quotient $G/B$ may be identified with the variety of complete flags $V_0\subset V_1\subset\cdots\subset V_{n}$ in $\mathbb C^n$.  The group $K$ acts on this variety with finitely many orbits; these are parametrized by the set $I_n$ of involutions in the symmetric group $S_n$\cite{MO88,RS90}.  In more detail, let $(\cdot,\cdot)$ be the standard symmetric bilinear form on $\mathbb C^n$, with isometry group $K$.  Then a flag $V_0\subset\cdots\subset V_n$ lies in the orbit $\mathcal O_\pi$ corresponding to the involution $\pi$ if and only if the rank $r_{ij}$ of $(\cdot,\cdot)$ on $V_i\times V_j$ equals the cardinality
$\pi_{ij}:=\#\{k: 1\le k\le i, \pi(k)\le j\}$ for all $1\le i,j\le n$.  

We use the same definition of pattern avoidance for permutations as in \cite{M10}, decreeing that $\pi=\pi_1\ldots\pi_n$ (in one-line notation) includes the pattern $\mu=\mu_1\ldots\mu_r$ if there are indices $i_1< i_2 <\cdots< i_r$ permuted by $\pi$ such that $\pi_{i_j}>\pi_{i_k}$ if and only if $\mu_j>\mu_k$.  We say that $\pi$ avoids $\mu$ if it does not include $\mu$.  (The more classical definition of pattern inclusion of Billey and others \cite{B00} would not require that
$\pi$ permute the indices $i_j$.  Thus by our definition the involution 65872143 does not include the pattern 2143, even though the indices $2,1,4,3$ occur in that order in the involution, since they are not permuted by it.   We will say more about the classical definition later.)

There are well-known poset- and graph-theoretic criteria for rational smoothness of complex Schubert varieties due to Carrell and Peterson.  The poset criterion does not extend to our setting but the graph one does.  To state it we first recall that the partial order on $I_n$ corresponding to inclusion of orbit closures is the reverse Bruhat order \cite{RS90}.   Then $I_n$ is graded via the rank function
$$
r(\pi) ={\lfloor n^2/4\rfloor} - \sum_{i<\pi(i)} (\pi(i) - i - \#\{k\in\mathbb N: i<k<\pi(i), \pi(k)<i\})
$$
\noindent where $\lfloor n^2/4\rfloor$ denotes the greatest integer to $n^2/4$ and $r(\pi)$ equals the difference in dimension between $\mathcal O_\pi$ and $\mathcal O_c$, the unique closed orbit, corresponding to the involution $w_0=n\ldots1$ \cite{RS90}.  Let $I_\pi$ be the interval consisting of all $\pi'\le\pi$ in the reverse Bruhat order.  We make $I_\pi$ into a graph by decreeing that the vertices $\mu$ and $\nu$ in it are adjacent if and only if either $\nu = t\mu t\ne\mu$ for some transposition $t$ in $S_n$, or  $\nu=t\mu$ for some transposition $t$ in $S_n$ with $t\mu t = \mu$ and $m$ is even; write $\nu = t\cdot\mu$ if either of these conditions holds.  We similarly make the full poset $I_n$ into a graph as well.  Then a necessary condition for $\bar{\mathcal O}_\pi$ to be rationally smooth is that the degree of $w_0$ in $I_\pi$ must be $r(\pi)$; in fact, all vertices $w^{-1}w_0w$ conjugate to $w_0$ in $W$ and lying in $I_\pi$ must have this degree \cite[2.5]{Br99}.  We conjecture that this last condition is also sufficient; this has been checked for $n\le 9$.  

\section{The bad patterns}
We recast the necessary pattern avoidance condition above for rational smoothness in terms of pattern avoidance.

\newtheorem{theorem}{Theorem}
\begin{theorem}
With notation as above, the orbit $\mathcal O_\pi$ has rationally singular closure whenever $\pi$ contains one of the twenty-four bad patterns
$14325,21543,32154,\linebreak 154326,124356,351624,132546,426153,
153624,351426,1243576,2135467,\linebreak 2137654, 4321576,5276143,5472163,1657324,
4651327,57681324,65872143,\linebreak 13247856, 34125768,34127856,64827153$.  The same holds if $\pi$ contains the pattern $2143$, provided that there are an even number of fixed indices of $\pi$ between $21$ and $43$ (e.g. $\pi=21354687$, where the $43$ occurs in the last two indices of $\pi$).
\end{theorem}

\begin{proof}
One checks first that the degree of $w_0$ in $I_\pi$ is greater than $r(\pi)$ for any $\pi$ in the list above, {\sl except} for 2137654 and 4321576.  The closures of the orbits corresponding to these two permutations are also rationally singular, as follows by computing that the degree of a suitable conjugate of $w_0$ (indexed by 7643521 in the first case and 7635421 in the second) is greater than $r(\pi)$.  If $\pi$ is obtained from one of the bad patterns above other than 2143 by adding one or more fixed points, then one computes that the degree of $w_0$ is again more than $r(\pi)$, except for 2134765, 3214576, 2137564, and 4231576, and in these cases again the degree at a suitable conjugate of $w_0$ is too large.  Moreover, any involution $\pi'$ obtained from one of the above ones by adding one fixed point has the degree of $w_0$ bigger than $r(\pi')$.   The same holds if $\pi$ is obtained from 2143 by adding two or more fixed points, with an even number of them lying between 21 and 43.  If $\pi$ is obtained from 2143 by adding just one fixed point not lying between 21 and 43, then the unique vertex conjugate to $w_0$ having this fixed point has degree larger than $r(\pi)$.  Finally, if $\pi$ is obtained from one of the bad patterns by adding fixed points as above and then pairs of flipped indices, then one argues as in the proof of the Lemma in \cite{M10} that some vertex in $I_\pi$ has degree greater than $r(\pi)$.  Hence in all cases $\bar{\mathcal O}_{\pi}$ is rationally singular.
\end{proof}

We specialize to the even case $n=2m$ in our next result.

\newtheorem*{theorem2}{Theorem 2}
\begin{theorem2}
Assume that $n=2m$ is even.  The orbit $\mathcal O_\pi$ has rationally smooth closure if and only if the degree of $w_0$ in $I_\pi$ is $r(\pi)$.
\end{theorem2}

\begin{proof}
Set $\mathcal O=\mathcal O_\pi$.  We have already noted that the degree condition is necessary, so suppose that it is satisfied.  We begin by constructing a slice of $\bar{\mathcal O}$ to $\mathcal O_c$ at a particular flag, as follows.  Fix a basis $(e_i)$ of $\mathbb C^{2m}$ such that $(e_i,e_j)=1$ if $i+j=2m+1$ and $(e_i,e_j)=0$ otherwise, where as above $(\cdot,\cdot)$ is the symmetric form.  Let $(a_{ij})$ be a family of complex parameters indexed by ordered pairs $(i,j)$ satisfying either $i\le m<j$ or $m<i<j$.  We assume that $a_{ij} = a_{2m+1-j,2m+1-i}$ if $i\le m<j$  but otherwise put no restrictions on the $a_{ij}$.  Define a basis $(b_i)$ of $\mathbb C^{2m}$ via 

\[ b_i = \begin{cases} e_i + \displaystyle\sum_{j=m+1}^{2m} a_{ij} e_j & \text{if $i\le m$}\\
e_i + \displaystyle\sum_{j=i+1}^{2m} a_{ij} e_j & \text{otherwise}
\end{cases}
\]

\noindent Then the Gram matrix $G := (g_{ij} = (b_i,b_j))$ of the $b_i$ relative to the form satisfies
\[
 g_{ij} = \begin{cases} 2a_{i,2m+1-j} & \text{if $i\le j\le m$}\\
g_{ji} & \text{if $j<i\le m$}\\
a_{j,2m+1-i} & \text{if $i<m<j<2m+1-i$}\\
1 & \text{if $i\le m<j= 2m+1-i$}\\
g_{ji} & \text{if $j\le m<i$}\\
0 & \text{otherwise}
\end{cases}
\]
\noindent Thus the matrix $G$ is symmetric and has zeroes below the antidiagonal from lower left to upper right.  The antidiagonal entries are all 1.  Now one checks that the set $\mathcal S'$ of all flags $V_0\subset\ldots\subset V_{2m}$ where $(b_i)$ runs through all bases obtained as above from the $a_{ij}$ and $V_i$ is the span of $b_1,\ldots b_i$ is a slice of $G/B$ to $\mathcal O_c$ at the flag $f_c$ corresponding to the basis $(e_i)$, which in turn corresponds to the point $P$ where all $a_{ij}=0$  \cite[2.1]{Br99}.  Intersecting $\mathcal S'$ with $\bar{\mathcal O}$ we get a slice $\mathcal S$ of $\bar{\mathcal O}$ to $\mathcal O_c$ at $P$ in the sense of Brion \cite[2.1]{Br99}, defined by the vanishing of certain minors in the Gram matrix $G$.   It is known (\cite{Br99}) that $\mathcal S$ is rationally smooth (resp. smooth) at $P$ if and only if it is rationally smooth (resp. smooth) everywhere, or if and only if $\bar{\mathcal O}$ is rationally smooth (resp. smooth) everywhere.   (This construction works with minor modifications for odd $n$ as well).

We now show that $\mathcal S$ is rationally smooth at $P$ by verifying the conditions of \cite[1.4]{Br99}.  Actually we construct a rationally smooth slice $\mathcal S''$ for a variety $\mathcal V$ containing $\bar{\mathcal O}_\pi$, with no hypothesis on the degree of $w_0$ in $I_\pi$; if this hypothesis holds, then $\mathcal V$ coincides with $\bar{\mathcal O}_\pi$ and we may take $\mathcal S'' =\mathcal S$.  For each conjugate $v=t\cdot w_0$ of $w_0$ by a transposition $t$ with $v\not\le\pi$, write $v$ as $v_1\ldots v_{2m}$ in one-line notation.  Let $i$ be the smallest index such that if $\pi_1\ldots \pi_i$ is rearranged in ascending order as $\pi_1'\ldots\pi_i'$ and similarly $v_1\ldots v_i$ is rearranged as $v_1'\ldots v_i'$, then $\pi_j'>v_j'$ for some $j\le i$.  Then there is some $k$ such that there are more indices $\ell\le i$ (say $n_k$ of them) with $v_\ell\le k$ than indices $m\le i$ with $\pi_m\le k$.  Let $j$ be the smallest index with $v_j\ne 2n+1-j$.  Set $r:=2n+1-k+n_k-2$.  If there are fewer than $n_k$ indices less than or equal to $k$ among $\pi_1\ldots\pi_r$, then the minor of $G$ consisting of those entries in rows $j,r-(n_k-2),\ldots,r-1,r$ and columns $\min(v_j,k-(n_k-1)),k-(n_k-2),\ldots, k-1,k$ vanishes on $\bar{\mathcal O}_\pi$; one variable in this minor occurs to the first power and is not multiplied by any other variable.  Otherwise the minor consisting of those entries in rows $j,k-(n_k -2),\ldots,k$ and columns $j,k-(n_k-2),\ldots,k$ (or just row and column $k$, if $n_k=1$) vanishes on $\bar{\mathcal O}_\pi$ and may involve certain variables quadratically.  Define the slice $\mathcal S''$ by the simultaneous vanishing of these minors and let $\mathcal V$ be the corresponding subvariety of $G/B$, which contains $\bar{\mathcal O}_\pi$.

Define an action of the $m$-torus $T=\mathbb T^m$ on the matrix $G$ by multiplying the first $m$ rows and columns by $t_1,\ldots,t_m$, respectively, while multiplying the last $m$ rows and columns by $t_m^{-1},\ldots,t_1^{-1}$, respectively; this action preserves the 1s on the antidiagonal and the vanishing of the minors that define the slice $\mathcal S''$. ($T$ is just a maximal torus of $K$.) Then the weights of $T$ occurring in the tangent space at $P$ of the big slice $\mathcal S'$ are those of the form $2e_i,e_i+e_j$, or $e_i-e_j$ for some $1\le i<j\le m$ and all occur with multiplicity one.  They all lie on one side of a hyperplane and $P$ is an attractive fixed point of both $\mathcal S''$ and $\mathcal S'$.   The subtori $ T'$ of $T$ of codimension one such that the fixed point subvariety $(\mathcal S'')^{T'}$ of the slice $\mathcal S''$ under the $T'$-action contains more than one point correspond exactly to the conjugates $t\cdot w_0$ of $w_0$ not lying below $\pi$.  Hence conditions \cite[1.4(ii),(iii)]{Br99} are satisfied for $\mathcal S''$ and $\mathcal V$.   As for \cite[1.4(i)]{Br99}, we find that by repeatedly slicing the slice $\mathcal S''$ at nonzero values of variables occurring in it, we are led to a product of varieties defined by quadratic equations in its variables, which is easily seen to be rationally smooth away from the origin.  (Here is where the evenness of $n$ is crucial; the varieties in question are not rationally smooth away from the origin if $n$ is odd.)  Thus $\mathcal S''$ and $\mathcal V$ are both rationally smooth.  If the degree condition holds on $w_0$, then by counting dimensions we see that $\mathcal V$ coincides with $\bar{\mathcal O}_\pi$, so it too is rationally smooth.

\end{proof}

\section{Main result}

\newtheorem*{theorem3}{Theorem 3}
\begin{theorem3}
If  $\pi$ avoids the bad patterns of Theorem 1, then all conjugates of $w_0$ in $I_\pi$ have degree $r(\pi)$, so that $\bar{\mathcal O}_\pi$ has rationally smooth closure if $n$ is even.
\end{theorem3}

\begin{proof}
Suppose first that $n$ is even.  We first show that $w_0$ has degree $r(\pi)$.  We have seen that the neighbors $\nu$ adjacent to any $\mu\in I_n$ (adjacent to $\mu$) take the form either $\nu=t\mu t$ or $\nu=t\mu$, for some transposition $t$; accordingly we say that $v$ is of {\sl type 1} (resp.\ {\sl type 2}) if it takes the first (resp.\ the second) form.  Now recall that the poset $I_n'$ of fixed-point-free involutions in $S_n$ with the reverse Bruhat order parametrizes the poset of Sp$_{2m}$-orbits in $G/B$, ordered by inclusion of closures; moreover the rank functions for $I_n$ and $I_n'$ coincide.  In \cite{M10} we characterized the fixed-point-free involutions $\pi$ such that the degree of $w_0$ in the order ideal $I_\pi'$ of $I_n'$ equals $r(\pi)$.  We first claim that there is a unique minimal fixed-point-free involution $f(\pi)$ lying above $\pi$ in the usual Bruhat order (so below it in the reverse one).  Indeed, if $\pi$ is fixed-point-free, then $f(\pi) = \pi$; otherwise let the indices fixed by $\pi$ be $i_1,\ldots,i_{2k}$ in increasing order, so that $\pi$ does not fix any index between $i_j$ and $i_{j+1}$ for any $j\le 2k-1$.  Now for every $j$ look at the pairs of indices both lying between $i_{2j-1}$ and $i_{2j}$ and flipped by $\pi$.  We say that such a pair $(i,\ell)$  with $i<\ell$ encapsulates another one $(j,k)$ if $i<j<k<\ell$.  Let $(\ell_{j1},\ell_{j1}'),\ldots,(\ell_{jm},\ell_{jm}')$ enumerate all such pairs not encapsulating other ones, labelled so that $\ell_{ji}<\ell_{ji}'$ for all $i$ and
$\ell_{ji}'<\ell_{j(i+1)}'$ for $i\le m-1$.  Replace $i_{2j-1},\ell_{j1}',\ldots,\ell_{jm}'$ in the one-line notation of $\pi$ by $\ell_{j1'},\ldots,\ell_{jm}',i_{2j}$, respectively, replace $\ell_{j1},\ldots,\ell_{jm},i_{2j}$ by $i_{2j-1},\ell_{j1},\ldots,\ell_{jm}$, respectively, and leave all other indices unchanged.  This gives the one-line notation of $f(\pi)$.  Thus for example if $\pi =16754238$, then $f(\pi) = 56781234$.  Now the type 1 neighbors(=adjacent vertices) of $w_0$ lying in $I_\pi$ are the same as those lying in $I_{f(\pi)}$.  We can read off the number $t(\pi)$ of type 2 neighbors of $w_0$ lying in $I_\pi$ from the one-line notation of $\pi$, as follows.  Let $i_1$ be the smallest index such that $\pi(i_1)\le i_i$ and let $i_2$ be the smallest index such that $\pi(n+1-i_2)\ge 2n+1-i_2$.  Let $i$ be the maximum of $i_1$ and $i_2$; then $t(\pi) = m+1-i$.  Now let $\pi$ be an involution of even length $2m$ that either appears in the above list or is obtained from an involution in this list by adding one fixed point.  One checks in all cases that  either the rank difference $r(\pi) - r(f(\pi))$ is less than $t(\pi)$ or $f(\pi)$ contains a bad pattern for $I_{2m}'$ (i.e., as a fixed-point-free involution, so that the degree of $w_0$ in $I_{f(\pi)}'$ is already too large).  Table 1 below lists the possibilities for $\pi$, making a representative choice among these possibilities if $\pi$ is obtained from a bad pattern of odd length by adding one fixed point.  We give first $\pi$, then its rank $r(\pi)$, then $f(\pi)$, then its rank $r(f(\pi))$ the number $t(\pi)$ of type 2 neighbors of $\pi$, and finally the difference between $r(f(\pi))$ and the degree of $w_0$ in $I_{2m}'$ whenever this difference is nonzero.
\vfil\eject
\vskip .5in
\begin{tabular}{| c | c | c | c | c | c |}\hline
$\pi$ & $r(\pi)$ & $f(\pi)$ & $r(f(\pi))$ & $t(\pi)$ & difference\\
\hline
2143 & 2& 2143 & 2 & 1 & \\ \hline
143256 & 7 & 341265 & 5 & 3 & \\ \hline
215436 & 6 & 215634 & 5 & 2 & \\ \hline
321546 & 6 & 351624 & 4 & 2 & 1\\ \hline
154326 & 5 & 456123 & 3 & 3 & \\ \hline
124356 & 8 & 214365 & 6 & 3 & \\ \hline
351624 & 4 & 351624 & 4 & 1 & 1\\ \hline
132546 & 7 & 351624 & 4 & 3 & 1\\ \hline
426153 & 4 & 456123 & 3 & 2 & \\ \hline
153624 & 5 & 351624 & 4 & 1 & 1\\ \hline
351426 & 5 & 351624 & 4 & 1 & 1\\ \hline
12435768 & 14 & 21437856 & 11 & 4 & \\ \hline
21354678 & 14 & 21563478 & 11 & 3 & 1\\ \hline
21376548 & 11 & 21678345 & 9 & 3 & \\ \hline
43215768 & 11 & 43218765 & 8 & 2 & 2\\ \hline
52761438 & 8 & 56781234 & 6 & 3 & \\ \hline
54721638 & 8 & 54721836 & 7 & 1 & 1 \\ \hline
16573248 & 9 & 56781234 & 6 & 4 & \\ \hline
46513278 & 9 & 46513287 & 8 & 1 & 1\\ \hline
57681324 & 5 & 57681324 & 5 & 0 & 1\\ \hline
65872143 & 4 & 65872134 & 4 & 0 & 1\\ \hline
34127856 & 10 & 34127856 & 10 & 2 & 1\\ \hline
64827153 & 5 & 64827153 & 5 & 1 & 1\\ \hline
13247856 & 12 & 34127856 & 10 & 2 & 1\\ \hline
34125768 & 12 & 34127856 & 10 & 2 & 1\\ \hline
\end{tabular}
\vskip .5in
\noindent Given an arbitrary involution $\pi$ for which the degree $d_{w_0}$ of $w_0$ in $I_\pi$ is too large, either the degree of $w_0$ in $I_{f(\pi)}'$ must already be too large (forcing $f(\pi)$ to contain one of the seventeen bad patterns of \cite{M10}) or $t(\pi)$ must be larger than the rank difference $r(\pi) - r(f(\pi))$ (or both).  The value of $t(\pi)$ is determined by the indices $i_1,i_2$ attached to $\pi$ above.  Bearing in mind the recipe for computing $f(\pi)$ from $\pi$ and the list of bad patterns in \cite{M10} (none of which has length larger than eight) we see that we can replace any $\pi$ for which $d_{w_0}$ is too large by an involution it includes of length at most eight with the same property (including the indices $i_1$ and $i_2$, two fixed points of $\pi$ with a pair of flipped indices lying between them, and that pair of indices).  But the above patterns capture all instances where this happens for $m=4$ (as one sees by examining all the possibilities, using for example the computation of the Kazhdan-Lusztig-Vogan polynomials defined in \cite{V83} for $GL_8\mathbb R$ provided by the ATLAS software, available at www.liegroups.org).  Thus any such $\pi$ includes a pattern in the above list, as desired.

Now we show that all conjugates of $w_0$ in $I_\pi$ also have degree $r(\pi)$.   By \cite[Theorem 2]{M10} the number of type 1 neighbors of any conjugate of $w_0$ lying below $f(\pi)$ in $I_\pi$ is no greater than $r(f(\pi))$, provided this holds for $w_0$.  One checks from the above formula for the number of type 2 neighbors of $w_0$ {\sl not} lying below $\pi$ that this number can only increase if $w_0$ is replaced by a conjugate of itself (i.e., by another fixed-point-free involution), so the degree of a conjugate of $w_0$ is bounded above by $r(\pi)$ whenever the degree of $w_0$ is, as desired.

If instead $n=2m+1$ is odd, then a similar but more complicated argument works.  Recall first that the neighbors of $\mu\in I_\pi$ all take the form $\nu=t\mu t$ for some transposition $t$ not commuting with $\mu$ (transpositions $t$ commuting with $\mu$ no longer give rise to adjacent vertices in this case).  We say that $\nu$ is of type 1 if the transposition $t$ does not involve the middle index $m+1$ and of type 2 if it does involve this index. Now given $\pi\in I_n$ there is a unique smallest $f(\pi)$ lying above $\pi$ in the Bruhat order among involutions fixing the index $m+1$ but no other.  To construct $f(\pi)$, assume first that $\pi$ already fixes $m+1$; then we just apply the above recipe for $f$ to $\pi$ restricted to the other indices $1\ldots,m,m+2,\ldots,2m+1$, decreeing at the end that $f(\pi)$ also fix $m+1$.  Now assume that $\pi(m+1)<m+1$; if this is not the case, conjugate $\pi$ by $w_0$ to make this hold, apply the following recipe, and then conjugate by $w_0$ again.  Denote by $i_1,i_2$ the largest fixed point of $\pi$ less than $m+1$ (if there is one) and the smallest fixed point of $\pi$ larger than $m+1$ (if there is one).  Enumerate the pairs of indices flipped by $\pi$ not encapsulating other pairs for which the larger index is greater than or equal to $m+1$ and the less than $i_2$ (if it exists) as $(\ell_1,\ell_1'),\ldots,(\ell_m,\ell_m')$ with $\ell_i<\ell_i'$ and the $\ell_i'.$ in increasing order, as in the previous recipe.  If $i_2$ exists, let $T$ consist of the set of $\ell_i'$ together with $i_2$; if $i_2$ does not exist, let $T$ consist of the $\ell_i'$ together with $\ell_m$.  Likewise let $S$ consist of the $\ell_i$ together with $i_1$ if it exists.  Define a new involution $\pi'$ by declaring that it fix the smallest index $m+1$ in $T$, flip the next smallest index of $T$ with the smallest one in $S$, and so on, finally flipping the largest index in $S$ with that in $T$ (if $i_2$ does not exist), or fixing the largest index in $S$ (if $i_1$ and $i_2$ both exist).  Other indices have the same image under $\pi'$ as under $\pi$.  Then $\pi'$ fixes $m+1$ and lies above $\pi$.  Applying the above recipe to $\pi'$, we get $f(\pi') = f(\pi)$.  Thus for example if $\pi = 13245$, then $S$ consists of the indices 1 and 2, while $T$ consists of 3 and 4; here $i_1 = 1,i_2 = 4$.  Then $\pi'$ fixes 3, flips 1 and 4, and fixes 2, whence $\pi' = 42315$; finally $f(\pi') = f(\pi) = 45312$.

We also have an analogous formula to the one above for the number $t(\pi)$  of type 2 neighbors of $w_0$ in $I_\pi$.  Let $i_1$ be the smallest index such that $\pi(i_1)\le i_1$ and $\pi(j)\ge m+1$ for some $j\ge n+1-i_1)$; similarly let $i_2$ be the smallest index such that $\pi(n+1-i_2)\ge n+1-i_2$ and $\pi(j)\le m+1$ for some $j\le i_2$.  Then the number $t(\pi) = n+1-i_1-i_2$.  A similar argument to the one above shows that avoiding the above list of bad patterns is sufficient to guarantee that the degree of $w_0$ in $I_\pi$ is $r(\pi)$; similarly any conjugate of $w_0$ in $I_\pi$ fixing $m+1$ has degree $r(\pi)$.

We now consider vertices in $I_\pi$ fixing a single index other than $m+1$.  For $1\le i\le n$, denote by $w_0^{(i)}$ the unique involution whose one-line notation has $i$ in the $i$th position and the other indices listed in decreasing order.  Then either $w_0^{(i)}$ lies in $I_\pi$, or else no vertex in $I_\pi$ fixes $i$ and no other index.  In the former case there is a unique smallest vertex $f^{(i)}(\pi)$ lying above $\pi$ in the usual Bruhat order, constructed as above by replacing the index $m+1$ throughout by $i$.  We construct it as above, replacing the index $m+1$ throughout by $i$, except that if $i$ is not the smallest index in the set $T$, then $w_0^{(i)}$ does not lie above $\pi$, so that no involution fixing $i$ alone lies above $\pi$ and $f^{(i)}(\pi)$ is undefined.  We argue as above that avoiding the bad patterns of Theorem 1 implies that $w_0^{(i)}$ has degree $r(\pi)$ whenever it lies in $I_\pi$ and that all conjugates of it fixing only the index $i$ have this degree as well.
\end{proof}

The condition that $w_0$ alone have degree $r(\pi)$ is {\sl not} sufficient for rational smoothness if $n$ is odd, as the examples $\pi = 2137654$ (cited above) and $\pi=21435$ show.  However we have 

\newtheorem*{conjecture4}{Conjecture 4}
\begin{conjecture4}
For $n$ odd, the orbit closure $\bar{\mathcal O}_\pi$ is rationally smooth if and only if all conjugates of $w_0$ in the order ideal $I_\pi$ have degree $r(\pi)$, or if and only if $\pi$ avoids all bad patterns.
\end{conjecture4}

We also have

\newtheorem*{conjecture5}{Conjecture 5}
\begin{conjecture5}
Avoiding all of the above patterns in the sense of this paper is equivalent to avoiding the same patterns in the classical sense
\end{conjecture5}

For example, we observed above that the involution 65872143 avoids the bad pattern 2143 in the sense of this paper, but it is itself another bad pattern, so does not correspond to an orbit with rationally smooth closure.  (We should mention that for the purposes of this conjecture, any involution with four indices permuted according to the pattern 2143 is counted as including that pattern, regardless of the number of fixed points between the 21 and the 43).  Probably this conjecture can be checked by a computer without too much trouble. 

 Finally, we note that in our setting, unlike that of \cite{H12} and \cite{M10} smoothness and rational smoothness of orbit closures are not equivalent.  The orbit closure corresponding to the involution 1324 is rationally smooth but not smooth.  We hope to study smoothness of orbit closures in a future paper.

\end{document}